\documentclass[12pt,a4paper]{amsart}
\usepackage{amsthm}
\usepackage{amsmath}
\usepackage{amsfonts}
\usepackage{amssymb}
\usepackage[left=2.8cm, right=2.8cm, top=3.0cm, bottom=3.0cm]{geometry}
\usepackage{color}
\usepackage[arrow, matrix, curve]{xy}
\usepackage[colorlinks=true,linkcolor=green]{hyperref}
\makeatletter
\renewcommand*{\eqref}[1]{
	\hyperref[{#1}]{\textup{\tagform@{\ref*{#1}}}}}
\makeatother

\newcommand{\inj}{\mathrm{inj}}

\newcommand{\eps}{\theta}

\newcommand{\beq}{\begin{equation}}
\newcommand{\beqn}{\begin{equation*}}

\newcommand{\eeq}{\end{equation}}
\newcommand{\eeqn}{\end{equation*}}

\newcommand{\R}{\mathbb{R}}

\newcommand{\n}{\mathbb{N}}

\newcommand{\tperp}{T^{\perp,v}_pM}
\newcommand{\z}{\mathbb{Z}}
\theoremstyle{plain}
\newtheorem{theorem}{Theorem}

\newtheorem{lemma}{Lemma}

\theoremstyle{definition}

\theoremstyle{remark}
\newtheorem{remark}{Remark}
\newtheorem{number-env}[theorem]{}
\title{Simple closed geodesics in dimensions $\ge 3$} 
\author{Hans-Bert Rademacher}
\address{Mathematisches Institut, 
Universit{\"a}t Leipzig, D--04081 Leipzig, Germany}
\email{rademacher@math.uni-leipzig.de}
\urladdr{www.math.uni-leipzig.de/\symbol{126}rademacher}
\date{2023-08-04, revised 2023-08-09}
\begin{document}
\begin{abstract}
	We show that for a generic Riemannian
	or reversible Finsler metric on
	a compact differentiable manifold $M$ of dimension
	at least three all closed geodesics are simple 
	and do not intersect each other.
	Using results by Contreras~\cite{C2010}
	\cite{C2011} this shows that for a
	generic Riemannian metric
	on a compact and simply-connected manifold
	all closed geodesics are simple and the number
	$N(t)$ of geometrically distinct closed geodesics
	of length $\le t$ grows exponentially. 
\end{abstract}
\maketitle
\baselineskip 18pt
\section{Results}
On a compact differentiable manifold 
$M$ 
endowed with
a reversible Finsler metric $f$ 
the corresponding norm 
of a tangent vector $v$ is given by 
$\|v\| = f(v).$ 
In the particular case of a Riemannian metric
$g$ we have $\|v\|^2=g(v,v).$
The Finsler metric is called \emph{reversible}
if $f(-v)=f(v)$ for all tangent vectors $v.$
With a closed geodesic
$c: S^1=\R/\z\longrightarrow M$
on a differentiable manifold equipped with
a Finsler metric $f$
the \emph{iterates} $c^m$ defined by
$c^m(t)=c(mt)$ are closed geodesics, too.
We call a closed geodesic
$c$ \emph{prime,} 
if there is no $m >1$ such that 
$c=c_1^m$ for a closed curve $c_1.$
For a 
a reversible Finsler metric
we call closed geodesics 
$c_1,c_2$ \emph{geometrically equivalent,}
if their traces $c_1(S^1)=c_2(S^1)$ coincide.
Otherwise we call the closed geodesics $c_1,c_2$
\emph{geometrically distinct.}
The orthogonal group $\mathbb{O}(2)$ acts
canonically
on the parameter space $S^1$ of loops, hence
also on the free loop space
$\Lambda M$ 
of the manifold $M$
by isometries leaving the 
energy functional
$
E: \Lambda M \longrightarrow \R\,;\, 2 E(\sigma)=
	\int_0^1 \|\sigma'(t)\|^2\,dt
$
invariant,
cf. \cite[Sec.2.3]{Kl} for the Riemannian case and
\cite[Sec.2]{Ra2017} for the Finsler case.
The closed geodesics are the critical points of 
the energy functional.
For an arbitrary closed geodesic $c$ there
is a prime closed geodesic $c_1$, such that all
geometrically equivalent closed geodesics
are of the form $\mathbb{O}(2).c_1^m, m\ge 1.$
A prime closed geodesic $c$ 
of a reversible Finsler metric
can have at most finitely many 
self-intersections, i.e. there are 
parameter values
$s,t\in S^1, s\not=t$
with $c(s)=c(t),$ 
cf. Lemma~\ref{lem:intersection}.
In this case $c'(s), c'(t)$ are linearly 
independent. We say a prime closed geodesic
is \emph{simple}, if it does not have self-intersections,
i.e. if the map $c: S^1\longrightarrow M$ is injective.
We say that two closed geodesics $c_1,c_2$ 
do not intersect
if their traces are disjoint, i.e. the 
intersection set
$I(c_1,c_2)=c_1(S^1)\cap c_2(S^1)=\emptyset$
is empty.

The main result of this paper is the following
\begin{theorem}
	\label{thm:one}
	Let $M$ be a compact differentiable manifold
	of dimension $n \ge 3.$
	For a $C^r$-generic Riemannian metric
	with $r \ge 2,$
	resp. a $C^r$-generic reversible Finsler metric
	with $r\ge 4,$
	all prime closed geodesics are simple, and 
	geometrically distinct 
	closed geodesics do not intersect each other.
\end{theorem}
The proof uses the \emph{bumpy metrics theorem} for 
Riemannian metrics due to Abraham~\cite{Ab}
with a detailed proof given by Anosov~\cite{An}.
The corresponding result for Finsler metrics
has been obtained by Taimanov and the author,
cf. \cite[Thm.3]{RT2020}. 
We state in Theorem~\ref{thm:rev} that the bumpy metrics
theorem also holds for reversible
Finsler metrics.

And we prove the following perturbation results:
Let $c:[0,a]/\{0,a\}=\R/(a \z) \longrightarrow M$ be a 
prime closed geodesic of
a Riemannian metric $g$ with $p=c(0)$ 
parametrized by arc length and hence of length $a.$
In Lemma~\ref{lem:one} we show that for an arbitrary small
$\eta>0$ we can define a one-parameter family of Riemannian
metrics $g^{(s)}, s \in [0,\delta]$ 
for some $\delta >0$
with $g^{(0)}=g$ such that the metrics
$g^{(s)}$ and $g$ only differ in an arbitrary small
tubular neighborhood $U$ of the geodesic segments
$c:[-2\eta,-\eta]\longrightarrow M$
and $c:[\eta,2\eta]\longrightarrow M$
and such that the metric $g^{(s)}$ for
$s>0$ has a closed geodesic $c_s:[0,a]/\{0,a\}\longrightarrow M$
parametrized by arc length which coincides with $c$ 
on the interval $[2\eta,a-2\eta].$ Furthermore the geodesic segment
$c_s|[-\eta,\eta]$ has positive distance from the 
geodesic segment $c|[-\eta,\eta].$
This local perturbation argument can be used 
on manifolds of dimension $n \ge 3$ to 
perturb away self-intersection and intersection points
of distinct closed geodesics, respectively, as we show in
Lemma~\ref{lem:perturbation_several}.
And it is also stated in
Lemma~\ref{lem:perturbation_several} that the analogous result holds for reversible Finsler metrics. 

We can combine the genericity statement 
of Theorem~\ref{thm:one} with other
genericity statements. Let $N(t)$ 
be the number of
geometrically distinct, closed geodesics
of length $\le t.$
The author has shown in~\cite{Ra94} that a 
$C^2$-generic 
Riemannian metric on a compact and simply-connected
manifold is \emph{strongly bumpy} and
carries infinitely many geometrically distinct closed
geodesics. This result has been used by 
Contreras~\cite{C2010}, 
\cite{C2011} to show that
for an open and dense set of
Riemannian metrics on a compact
and simply-connected manifold 
with respect to the $C^2$-topology
the geodesic flow 
contains a 
non-trivial basic hyperbolic set. 
In particular this implies
that $N(t)$ grows exponentially.
Hence we obtain  from Theorem~\ref{thm:one} and
\cite{C2010}, \cite{C2011}:
\begin{theorem}
	Let $M$ be a compact and simply-connected manifold
	of dimension $n\ge 3.$
	For a $C^2$-generic Riemannian metric 	
	all prime closed geodesics are simple and do not intersect
	each other. Furthermore the number 
	$N(t)$ of closed geodesics of length $\le t$
	grows exponentially, i.e.
	$$
	\liminf_{t\to \infty} \left(\log N(t)\right)/t>0.$$	
\end{theorem}
For surfaces results are quite different,
cf. Remark~\ref{rem:surfaces}. For example a generic
Riemannian or reversible Finsler metric of positive
curvature on a two-dimensional sphere $S^2$
has only finitely many simple closed geodesics
but infinitely many geometrically distinct closed
geodesics with self-intersections.
Surveys for existence results for closed geodesics
are \cite{Ba}, \cite{Oa} and \cite{T2010}.\\[2ex]
{\sc Acknowledgement.} 
I am grateful to the referee for 
the helpful comments and suggestions.
\section{Perturbing a single geodesic segment}
\label{sec:perturb}
For the proof of our perturbation result 
we introduce at first geodesic coordinates in a tubular neigbhorhood of 
a geodesic segment. 
We first discuss in detail the case of a Riemannian
metric $g.$ Later we explain which changes
are necessary in case of a reversible Finsler metric $f.$

On a compact Riemannian manifold $(M,g)$
there is a positive
number $\widetilde{\eta}<\inj/3,$ where
$\inj$ is the 
\emph{injectivity radius,} such that 
the following holds for any
point $p \in M,$ any 
unit tangent vector $v \in T_p^1M,$ 
any $0<\eta\le \widetilde{\eta}$
and any 
sufficiently small

$\epsilon >0$ with $0 <7 \epsilon < \eta:$
Let $\tperp :=
\{x \in T_pM\,;\, \langle v, x\rangle=0\}$
be the orthogonal complement of the 
one-dimensional subspace generated by
$v.$
And for a linear subspace $W \subset T_pM$
define $B_{\epsilon}(W):=\{x\in W\,;\, \|x\|<\epsilon,\},$
hence
$B_{\epsilon}(\tperp) =
\{x \in T_pM; \|x\|<\epsilon, \langle x,v\rangle=0\}.$
Then  the restriction
\begin{equation}
	\label{eq:normal_exp}
	\exp^{\perp,v}: 
	(t,x)\in 
	[-2\eta,2\eta]
	\times
	B_{\epsilon}(T_p^{\perp,v}M) 
	\longrightarrow 
	\exp(t \nu (\exp_p(x))\in M
\end{equation} 
of the
normal exponential map 
is a diffeomorphism onto 
the tubular neighborhood 
$Tb_{v}(\eta,\epsilon)
=\exp^{\perp,v}
\left(
[-2\eta,2\eta]
\times
B_{\epsilon}\left(\tperp\right)
\right)
$ of the geodesic
$c_v:[-2\eta,2\eta]\longrightarrow M$ with
$c_v(0)=p, c_v'(0)=v.$
Here
$\nu: \Sigma_{p,v}(\epsilon)\longrightarrow
T^1M$
is the unit normal vector field 
defined on the 
local hypersurface $\Sigma_{p,v}(\epsilon)=
\exp_p\left(B_{\epsilon}
\left(T_p^{\perp,v}M\right)\right)$
with $\nu(p)=v.$ 

For $q \in M, r>0$ we denote
by $B_q(r)$ the geodesic ball around
$q$ of radius $r,$ i.e.
$B_q(r)=\{x \in M\,;\, d(q,x)\le r\}.$
Here $d$ is the distance induced by the
Riemannian metric $g.$
For statements about Riemannian geometry we
refer to~\cite[ch.1]{Kl}.
We define the \emph{spherical shell}
$A_p(\eta,\epsilon)=
B_p(\eta+7\epsilon)-B_p(\eta)$
around $p.$ And for $v\in T_p^1M$ we have the
tubular neighborhood
$Tb_{v}(\eta,\epsilon)$
and the sets 
$U^{-}_{v}(\eta,\epsilon)=
\exp_p^{\perp,v}
\left((-\eta-6\epsilon,-\eta)\times
B_{\epsilon}(\tperp)
\right)$
and
$U^{+}_{v}(\eta,\epsilon)=
\exp_p^{\perp,v}
\left((\eta,\eta+6\epsilon)\times
B_{\epsilon}(\tperp)
\right).$
Hence the union
$U_v(\eta,\epsilon)=
U_v^{-}(\eta,\epsilon)
\cap
U_v^+(\eta,\epsilon)$
is the disjoint union
of the two connected sets $U^{\pm}_v(\eta,\epsilon).$
And $U_v(\eta,\epsilon)
\subset
A_p(\eta,\epsilon).$
If now $w\in T_p^1M$ is a unit vector
orthogonal to $v$ we can choose a local
isometry 
\begin{equation}
	\label{eq:zeta}
	\zeta=\zeta_w:
	(B_{\epsilon}(\tperp),g_p)
	\longrightarrow
	D^{n-1}(\epsilon)
\end{equation}
between $B_{\epsilon}(\tperp)$
with the Riemannian metric $g_p,$
and an Euclidean disc $D^{n-1}(\epsilon)$ of radius
$\epsilon$ and dimension $(n-1)$
with $\zeta_{w}(w)=e_2.$
Here $(e_1,e_2,\ldots,e_n)$ is an orthonormal basis
for $\R^n$ 
with coordinates $x=\sum_{i=1}^nx_ie_i$
and $D^{n-1}(\epsilon)=\{(0,x_2,\ldots,x_n)
\in \R^n\,;\,\sum_{j=2}^nx_j^2<\epsilon\}.$
We obtain the following
diffeomorphism:
\begin{equation}
	\label{eq:xi}
	\xi=\xi_{v,w}: 
	[-2\eta,2\eta]
	\times 
	D^{n-1}(\epsilon)
	\longrightarrow Tb_v(\eta,\epsilon)\,;\,
	(t,y)	\longmapsto \exp^{\perp,v}(t \nu(\exp_p(\zeta_{w}^{-1}(y)))	\,.
\end{equation}
By definition 
for sufficiently small $\epsilon>0$
the curves with $y=const$ are geodesics
parametrized by arc length.
The subset
\begin{equation}
	\label{eq:pvw}
	P(v,w)=\{\exp^{\perp,v}(t\nu(\exp_p(sw)))\,;\,
	t\in [-2\eta,2\eta],-\epsilon<s<\epsilon\}		
\end{equation}
is a local surface defined in the tubular neighborhood 
$Tb_v(\eta,\epsilon)$ of the geodesic $c.$
\begin{remark}[Finsler case, orthogonal complement, exponential map]
	For facts about Finsler metrics we refer to
	the books~\cite{BCS} and \cite{Sh}. If $f:TM \longrightarrow \R$
	is a Finsler metric
	one obtains in each point $p\in M$ a
	whole family $g_v$ of Riemannian metrics 
	parametrized by unit tangent vectors
	$v \in T_p^1M.$ This Riemannian metric is defined
	for $x,y \in T_pM$ 
	as follows:
	\begin{equation}
		g_v(x,y)=\langle x,y \rangle_v=\frac{1}{2}\left. 
		\frac{\partial^2}{\partial s \partial t}\right|_{s=t=0}f^2(v+sx+ty)\,.		
	\end{equation}
	Then the \emph{Legendre transformation}
	\begin{equation}
		\mathcal{L}:TM \longrightarrow 
		T^*M\,;\, v \longmapsto \langle
		.,v\rangle_v
	\end{equation}
	is defined. The orthogonal complement
	of a unit vector $v \in T^1_pM$ can be defined
	as the kernel of the linear form
	$\mathcal{L}(v),$
	i.e. the orthogonal complement of $v$ with respect to
	the Riemannian metric $g_v.$
	The exponential mapping 
	$\exp_p:B_{\epsilon}(T_pM)\longrightarrow B_p(\epsilon)\subset M $
	restricted to tangent vectors of length $<\epsilon,$
	where $\epsilon$ is smaller than the injectivity radius,
	is a $C^1$-diffeomorphism, outside $0$ the map is
	$C^{\infty}.$
\end{remark}
\begin{lemma}
	\label{lem:one}
	Let $(M,g)$ be a compact Riemannian manifold, $p \in M$
	and $v,w \in T_p^1M$ two unit vectors orthogonal to each other,
	i.e. $\langle v,w\rangle=0.$ Let $\eta>0$ be 
	sufficiently small, i.e.
	$\eta \le \widetilde{\eta},$
	cf. the beginning of Section~\ref{sec:perturb}.
	Let $\epsilon >0$ satisfy 
	$0<6\epsilon <\eta$ and such that the map
	$\xi=\xi_{v,w}$ from Equation~\eqref{eq:xi} is a 
	diffeomorphism and $c(t)=\xi_{v,w}(t,0)
	=\exp(tv), -2\eta\le t \le 2\eta$ 
	is a geodesic parametrized by arc length.
	
	Then there is a smooth one-parameter family of Riemannian metrics
	$g^{(s)}, s\in [0,\delta]$ for some
	sufficiently small $\delta>0$
	with $g=g^{(0)}$ such that 
	the metric $g^{(s)}$ has a geodesic
	$c_s(t)=\xi(t,u_{s}(t),0,\ldots,0)$ 
	with 
	a scalar function 
	$u_s:[-2\eta,2\eta]\longrightarrow \R^{\ge 0}$
	satisfying
	$u_{s}(t)=0$ for all 
	$t \in (-2\eta,-\eta-4\epsilon)
	\cup(\eta+4\epsilon,2\eta)$
	and $u_{s}(t)=s$ for all $t \in [-\eta-2\epsilon,\eta+2\epsilon]$ 
	and sufficiently small $s>0,$
	and
	$u_{s}'(t)t\le 0,$ i.e. $u_{s}$ is monotone increasing for 
	$t\le -\eta$ and monotone decreasing for $t\ge\eta.$
	In particular, for $s>0$ the geodesic $c_s$ lies in
	the local surface $P(v,w)-\{p\}$ 
	and has the same length $2\eta$ as $c.$
	The metrics $g^{(s)}$ and $g$ coincide outside
	the set $U_v(\eta,\epsilon)=
	U_v^-(\eta,\epsilon)\cup
	U_v^+(\eta,\epsilon).$
\end{lemma}
\begin{proof}
	Define 
	for sufficiently small $\delta>0$
	the following vector field 
	$V(t,x_2,\ldots,x_n)$ for
	$(t,x) \in 
	[-2\eta,2\eta]\times 
	D^{n-1}(\epsilon)$
	with $x=(x_2,\ldots,x_n):$
	
	Let $V(t,x_2,\ldots,x_n)=(0, \psi(t,x),0,\ldots,0)$ where
	$(t,x)\in [-2\eta,2\eta]\times D^{n-1}\longmapsto
	\psi(t,x)\in [0,1]$ is a smooth function, 
	and
	$\psi(t,x_2,0, \ldots,0)=1$ for 
	$|t|\le \eta + 2\epsilon, 0\le x_2\le \delta$
	and $\psi(t,x)=0$ for $|t|\ge \eta+4\epsilon$ or
	$\|x\|\ge 2\delta.$
	In addition we assume $\psi(-t,x)=\psi(t,x)$
	for all $t,x.$ 
	This vector field determines 
	an one-parameter group of diffeomorphisms
	$$\widetilde{\Psi}^s: 
	[-2\eta,2\eta]\times
	D^{n-1}(\epsilon) 
	\longrightarrow
	[-2\eta,2\eta]
	\times 
	D^{n-1}(\epsilon) $$
	with $s \in \R;\widetilde{\Psi}^s(t,x)=(t,x)$ for 
	$|t|\ge \eta+4\epsilon$ or $\|x\|\ge \delta$
	and $\widetilde{\Psi}^s(t,0)=(t,s,0,\ldots,0)$ 
	for $s \in [0,\delta]$
	and $|t|\le \eta+2\epsilon.$ In addition
	$\widetilde{\Psi}^s(t,0)=(t,u_{s}(t),0,\ldots,0)$
	with a smooth and even function 
	$t\longmapsto u_{s}(t)$
	which is monotone increasing for
	$t<0$ and hence monotone decreasing for $t>0,$
	and satisfies $u_{s}(t)=s$ for $0< s\le \delta$ for all
	$t$ with $|t|\le \eta+2\epsilon$
	and $u_s(t)=0$ for all $t$ with $|t|\ge \eta+4\epsilon.$
	With the help of the diffeomorphism $\xi_{u,v}$ defined in
	Equation~\eqref{eq:xi} we obtain
	a one-parameter group of 
	diffeomorphisms
	$\Psi^s: M \longrightarrow M$ with
	$\Psi^s (y)=y$ for all 
	$y \in M-Tb_{v}(\eta,\epsilon)$
	and $s\ge 0,$
	and
	\begin{equation}
		\label{eq:diffeo-s}
		\Psi^s(y)=
		\xi_{v,w}\left(\widetilde{\Psi}^s
		\left(\xi_{v,w}^{-1}
		\left(y\right)\right)\right)
	\end{equation}
	for all $y \in Tb_v(\eta,\epsilon), s\in \R.$
	Hence for the geodesic $c(t)=\xi_{v,w}(t,0),
	t\in [-2\eta,2\eta]$ we obtain
	\begin{equation}
		\label{eq:csc}
		c_s(t)=\Psi^s(c(t))=\xi_{v,w}(t,u_{s}(t),0).
	\end{equation}
	And the curves $c_s=c_s(t),c=c(t)$ coincide on 
	$[-2\eta,-\eta-4\epsilon]
	\cup
	[\eta+4\epsilon,2\eta].$
	
	For a positive number $\eps>0$ 
	choose smooth cut-off functions
	$\beta_{\eps}, \beta_{\eta,\epsilon}:\R \longrightarrow [0,1]$
	with $\beta_{\eps}(-t)=\beta_{\eps}(t)$ and $
	\beta_{\eta,\epsilon}(t)=
	\beta_{\eta,\epsilon}(-t)$ for all $t \ge 0.$
	The function $\beta_{\eps}:[0,\infty]\longrightarrow \R$ 
	is monotone decreasing and satisfies 
	$\beta_{\eps}(t)=1$ for all $t\in [0,\eps],$ 
	and
	$\beta_{\eps}(t)=0$ for all $t\ge 2\eps.$
	The function
	$\beta_{\eta,\epsilon}:[0,\infty]\longrightarrow \R$ 
	satisfies:
	$\beta_{\eta,\epsilon}(t)=1$
	for all $t$ with $|t|\in [\eta+\epsilon,
	\eta+5\epsilon],$ and
	$\beta_{\eta,\epsilon}(t)=0$ for
	all $t$ with $|t|\le \eta$ or $|t|\ge \eta+
	6\epsilon.$
	And $\beta_{\eta,\epsilon}$ is monotone
	increasing on $[\eta,\eta+\epsilon]$ and monotone
	decreasing on $[\eta+5\epsilon,\eta+6\epsilon].$
	Then the smooth function
	$\alpha=\alpha_{\eta,\epsilon,\delta}:
	M \longrightarrow [0,1]$
	is defined as follows:
	$\alpha(\xi_{v,w}(t,x))=\beta_{\delta}(\|x\|)
	\beta_{\eta,\epsilon}(t)$ and 
	$\alpha(x)=0$ for $x \not\in Tb_v(\eta,\epsilon).$
	By definition of the function 
	$\alpha=\alpha_{\eta,\epsilon,\delta}$ it follows that
	$\alpha$ vanishes outside the set
	$U_v(\eta,\epsilon),$
	i.e. $\alpha(y)=0$ for $y \not\in U_v(\eta,\epsilon).$
	With this function we define the 
	smooth pertubation
	$g^{(s)}, s\in [0,s_0]$ of the metric 
	$g=g^{(0)}$
	by the following convex combination of
	metrics:
	\begin{equation}
		\label{eq:gs}
		g^{(s)}=(1-\alpha) g +\alpha (\Psi^{-s})^*(g)\,.
	\end{equation} 
	Here $g_1^{(s)}=(\Psi^{-s})^*(g)$ is the pull-back
	of the Riemannian metric $g$ via the diffeomorphism
	$\Psi^{-s}.$ By definition the mapping
	$\Psi^s: (M,g) 
	\longrightarrow (M,g^{(s)}_1)$ is an isometry
	between these Riemannian manifolds.
	We conclude that the geodesic
	$c: \R \longrightarrow M$ with $c(0)=p, c'(0)=v$
	with respect to the metric $g$
	is mapped onto the geodesic
	$c_s:\R \longrightarrow M$ defined above
	of the metric $g_1^{(s)}.$
	The restriction 
	$c_s:[-\eta-2\epsilon,\eta+2\epsilon]
	\longrightarrow M$
	is a geodesic 
	parametrized by arc length
	of the metrics $g$ and of the
	metric $g_1^{(s)}=(\Psi^{-s})^*(g).$ 
	Hence 
	$c_s:[-\eta-2\epsilon,\eta+2\epsilon]
	\longrightarrow M$
	is also a geodesic of the  convex
	combination $g^{(s)}$
	defined in Equation~\eqref{eq:gs}.
	This one can check easily since the energy
	$E_s(\gamma)$ with respect to $g^{(s)}$
	of an arbitrary smooth curve 
	$\gamma(t)=\xi_{v,w}(t,x(t)), t\in
	[t_0,t_1]\subset [-\eta-2\epsilon,\eta+2\epsilon]$
	with $x(t_0)=x(t_1)$
	satisfies
	$E_s(\gamma)\ge (t_1-t_0)/2$
	and the restriction
	$c_s|[t_0,t_1]$ satisfies
	$E_s(c_s|[t_0,t_1])=(t_1-t_0)/2.$
	This is a consequence of the
	\emph{Gau{\ss} Lemma,} cf.
	\cite[sec.1.9]{Kl} and the 
	following Remark~\ref{rem:gauss},
	since the curves 
	$t\in[-\eta-2\epsilon,\eta+2\epsilon]\mapsto \xi_{v,w}(t,x)
	\in M$ for fixed $x\in D^{n-1}$ are geodesics
	parametrized by arc length of the metric
	$g$ and of the metric $g_1^{(s)}.$
	Therefore the curve 
	$c_s:[-\eta-2\epsilon,\eta+2\epsilon]
	\longrightarrow M$ is locally the shortest
	connection of its endpoints with respect to
	the metric $g^{(s)},$ hence a geodesic.
	On the other hand the curve
	$c_s:[-2\eta,2\eta]\longrightarrow M$
	is a geodesic of $g_1^{(s)},$ since
	$\alpha(c_s(t))=1$ for $|t|
	\in[\eta+\epsilon,\eta+5\epsilon]$
	it also follows that 
	$c_s:[\eta+\epsilon,\eta+5\epsilon]
	\longrightarrow M$ as well
	as 
	$c_s:[-\eta-5\epsilon,-\eta-\epsilon]
	\longrightarrow
	M$ is also a geodesic of $g^{(s)}$ since
	it is a geodesic of $g^{(s)}_1.$
	In the line following 
	Equation~\eqref{eq:csc}
	we already have seen that the curves
	$c$ and $c_s$ coincide on
	$[\eta+4\epsilon,2\eta]$
	and $[-2\eta,-\eta-4\epsilon].$
	And since $\Psi(\xi_{v,w}(t,x))=\xi_{v,w}(t,x)$
	for $|t|\in[\eta+4\epsilon,2\eta]$ the
	metrics $g,g_1^{(s)}$ and $g^{(s)}$
	coincide on the set
	$\xi_{v,w}(([-2\eta,-\eta-4\epsilon]\cup[\eta+4\epsilon,2\eta])\times D^{n-1}).$
	Hence $c_s:[-2\eta,2\eta]\longrightarrow M$
	is a geodesic segment of the Riemannian metric
	$g^{(s)}.$
	
	We have shown that the
	smooth one-parameter family $g^{(s)}$
	of Riemannian metrics for sufficiently
	small $s$ carries a geodesic
	$c_s(t)$ with the following properties: 
	$c_s$ coincides outside $Tb_v(\eta,\epsilon)$
	with the geodesic $c$ of the metric $g$
	and $c_s(t)=\xi_{v,w}(t,u_{s}(t),0,\ldots,0)$ for
	$t\in [-2\eta,2\eta]$ and $u_{s}(t)=s$
	for $t \in [-\eta-2\epsilon,\eta+2\epsilon].$
	And $u_s(t)=0$ for $t \in [-2\eta,-\eta-4\epsilon]
	\cup [\eta+4\epsilon, 2\eta].$
	In particular
	$c_s([-2\eta,2\eta])\subset P(u,v)-\{p\},$
	i.e. $c_s$ lies in the local surface
	$P(u,v)$ and avoids the point $p$
	for $s>0.$
	
	The arguments 
	can be carried over to the case of a
	reversible Finsler metric. 
	In particular the 
	Finsler metric $f^{(s)}$
	corresponds to the convex combination
	of Riemannian metrics given
	in Equation~\eqref{eq:gs}:
	\begin{equation*}
		f^{(s)}=\sqrt{(1-\alpha) 
			f^2 +\alpha \left(\left(\Psi^{-s}\right)^*(f)\right)^2}	\,.
	\end{equation*}	
	This is a Finsler metric, for which 
	$	c_s(t)=\xi_{v,w}(t,0)=
	\xi_{v,w}(t,u_s(t),0,\ldots,0)$
	is a geodesic as in the Riemannian case.
	For the Gau{\ss} Lemma in 
	Finsler geometry cf.~\cite[Sec.6.1]{BCS}.
\end{proof}
\begin{remark}[Convex combination of Riemannian and Finsler metrics]
	\label{rem:gauss}
	(a) In the proof we used the following local statement
	in a Riemannian manifold. Let 
	$(t,x)\in [-\eta-2\epsilon,\eta+2\epsilon]\times
	D^{n-1}$ be local coordinates and let $g^{(0)},g^{(1)}$
	be Riemannian metrics 
	for which the $t$-lines are geodesics parametrized 
	by arc length starting orthogonally from the
	hypersurface $t=0.$
	These coordinates are also called \emph{geodesic parallel
		coordinates} based on the hypersurface $t=0.$
	Then
	the convex combination
	$$
	g^{*}=\alpha g^{(0)} +(1-\alpha) g^{(1)}
	$$
	for a smooth function 
	$(t,x) \in [-\eta+2\epsilon,\eta+2\epsilon]\times D^{n-1}
	\longmapsto\alpha=\alpha(t,x)\in [0,1]$ 
	is a Riemannian metric, for which also the
	$t$-coordinate lines are geodesics parametrized by
	arc length, i.e. the coordinates $(t,x)$ are also
	geodesic parallel coordinates 
	for the Riemannian metric $g^*$ based on the hypersurface
	$t=0.$ As indicated in the previous proof the argument
	is that the curve $t\in [t_0,t_1] \mapsto (t,x)\in [-\eta-2\epsilon,\eta+2\epsilon]\times
	D^{n-1}$ for a fixed $x$ is locally the shortest curve joining
	$(t_0,x)$ and $(t_1,x)$ as a consequence of the
	Gau{\ss} Lemma.
	
	The statement follows also if you write down the 
	line elements of the metrics $g^{(0)},g^{(1)}$
	by
	$$
	dt^2+\sum_{i,j=1}^{n-1} g^{(l)}_{ij}(t,x) dx^i dx^j, \, l=0,1
	$$
	with respect to the
	coordinates $(t,x_1,\ldots,x_{n-1})$
	for $t\in [-\eta-2\epsilon,\eta+2\epsilon].$
	Then the  line element of $g^*$ is of the form
	$$
	dt^2+\sum_{i,j=1}^{n-1} g^{*}_{ij}(t,x) dx^i dx^j,
	$$
	with 
	metric coefficients $g^{*}_{ij}(t,x)=\alpha(t,x)g^{(0)}_{ij}(t,x)+
	\left(1-\alpha(t,x)
	\right) g^{(1)}_{ij}(t,x).$
	Therefore the lines with $x=const$ are geodesics 
	of
	the metric $g^*.$

	\smallskip
	
	(b) The analogous statement for reversible Finsler metrics is
	the following: Let
	$(t,x)\in 
	[-\eta+2\epsilon,\eta+2\epsilon]\times D^{n-1}$ 
	be local
	coordinates and $f^{(0)},f^{(1)}$ be reversible Finsler metrics
	for which the $t$-lines are geodesics parametrized by arc length
	and starting orthogonally from the hypersurface $t=0.$
	Then the convex combination
	$$f^*=\sqrt{\alpha f^{(0)}+(1-\alpha)f^{(1)}}$$
	for a smooth function 
	$(t,x)\in [-\eta+2\epsilon,\eta+2\epsilon]
	\times D^{n-1}\longmapsto \alpha(t,x)\in [0,1]$
	is a reversible Finsler metric for which also the $t$-lines
	are geodesics parametrized by arc length. 
	This follows from the Gau{\ss} Lemma for Finsler metrics,
	cf.~\cite[Sec.6.1]{BCS}
\end{remark}
\section{Intersection of geodesics}
\label{sec:intersection}
The following statement is well known for Riemannian
metrics,
cf. for example \cite[Sec.2.3]{BG2010}, the statements 
carry over to the case of
reversible Finsler metrics, as we will show:
\begin{lemma}
	\label{lem:intersection}
	Let $g$ be a Riemannian metric resp.
	let $f$ be a reversible Finsler metric
	and $c, d: S^1\longrightarrow M$ be prime closed
	geodesics, which are geometrically distinct.
	
	\smallskip 
	
	(a) The set of \emph{double points,} resp.
	self-intersection points 
	$$DP(c):=\#\{c(t)\,;\, t\in S^1,\#c^{-1}(c(t))\ge 2\}$$
	of the closed geodesic $c$ is finite.
	
	\smallskip
	
	(b) The \emph{intersection} $I(c,d):=c(S^1)\cap d(S^1)$ is finite.	
\end{lemma}
\begin{proof}
	Since $c$ is an immersion and since $S^1$ is compact
	the set
	$$c^{-1}(c(t))=\{s \in S^1\,;\, c(s)=c(t)\}
	$$
	is finite for all $t \in S^1.$
	
	\smallskip
	
	(a) If the set $DP(c)$ is not finite then there exist
	sequences $s_j, t_j\in S^1$ converging to $s^*, t^*\in S^1$ 
	with $
	s_j\not=s^*;t_j\not=t^*, s_j\not=t_j$ for
	all $j$	and $c(t_j)=c(s_j)$ for all $j\ge 1.$
	Then we conclude $p:=c(t^*)=c(s^*).$ If $c'(t^*)=\pm c'(s^*)$
	then the closed geodesic is not prime. 
	Hence for $v=c'(s^*),w=c'(t^*)$ we have 
	$v\not=\pm w.$ Since $\|c'\|$ is constant, we obtain
	$\|v\|=\|w\|.$
	The exponential map
	$\exp_p:B_{\tilde{\eta}}(T_pM)
	\longrightarrow B_{\tilde{\eta}}(p)$
	is injective, we conclude:
	Since $c(t_j)=\exp_p((t_j-t^*)v)=
	c(s_j)=\exp_p((s_j-s^*)w)$ 
	for sufficiently large $j$ with
	$|t_j-t^*|\|v\|, |s_j-s^*|\|w\| <\inj$
	we conclude
	that $(t_j-t^*)v=(s_j-s^*)w$ holds, i.e.
	$v =\pm w,$ which is a contradiction.
	
	\smallskip
	
	(b) The argument is similar, if the set
	$I(c,d)$ is infinite, then there are
	sequences $s_j, t_j\in S^1$ converging to $s^*, t^*\in S^1$ 
	with $
	s_j\not=s^*;t_j\not=t^*, s_j\not=t_j$ for
	all $j$	and $c(t_j)=d(s_j)$ for all $j\ge 1.$
	Then we conclude $p:=c(t^*)=d(s^*).$ If $c'(t^*)=\pm d'(s^*)$
	then the closed geodesics $c,d$ are geometrically
	equivalent. 
	Hence for $v=c'(s^*),w=d'(t^*)$ we have 
	$v\not=\pm w.$ Since $\|c'\|$ is constant, we obtain
	$\|v\|=\|w\|.$
	Since $c(t_j)=\exp_p((t_j-t^*)v)=
	d(s_j)=\exp_p((s_j-s^*)w)$ we conclude
	for sufficiently large $j$ with
	$|t_j-t^*|\|v\|, |s_j-s^*|\|w\| <\inj$
	that $(t_j-t^*)v=(s_j-s^*)w$ holds, i.e.
	$v/\|v\| =\pm w/\|w\|,$ which is a contradiction.
\end{proof}

\begin{remark}[Self-intersection of closed geodesics
	for non-reversible Finsler metrics]
	Note that the last argument does not work for
	non-reversible Finsler metrics.
	If $\gamma(t)$ is a geodesic
	with $\gamma_v(0)=p, \gamma_v'(0)=v,$ 
	then the equation $\gamma_v(t)=\exp_p(tv)$
	only holds for $t\ge 0.$ In general
	$\gamma_v(-t)\not=\exp_p(-tv)=\gamma_{-v}(t).$
	Note that two closed geodesics 
	$c_1,c_2:S^1\longrightarrow M$
	of a non-reversible
	Finsler metric are 
	called \emph{geometrically equivalent} only
	if their traces $c_1(S^1)=c_2(S^1)$ 
	\emph{and} their
	orientations coincide.
	The Katok-examples 
	as non-reversible perturbations of the standard
	Riemannian metric on a sphere 
	yield metrics for which there are
	two 
	geometrically distinct
	closed geodesics $c_1,c_2:S^1\longrightarrow S^n$
	which have the same trace but different orientation
	and length. These metrics are explained in \cite{Zi}, 
	\cite[Sect.11]{Ra2004},  and
	\cite{Bryant}. Hence this is an example of a non-reversible
	Finsler metric for which there are geometrically distinct
	closed geodesics intersecting in an infinite number
	of points.	
	
\end{remark}
\section{Perturbing intersecting geodesic segments}	
\label{sec:perturbing-intersecting}
For a point $p\in M$ and $\eta,\epsilon >0$
with $6\epsilon <\eta$ we have defined the 
spherical shell
$A(\eta,\epsilon)
=\{x \in M\,;\, \eta<d(x,p)<\eta+7\epsilon\}$ 
around $p.$
For a geodesic segment
$c_j:[-2\eta,2\eta]\longrightarrow M, j=1,\ldots,N$ 
parametrized by arc length
with $p=c_j(0)$ and $v_j:=c_j'(0)$
we recall the definition of the sets
$U_j(\eta,\epsilon)=U_{v_j}(\eta,\epsilon)$
which are subsets of the
of the spherical shell $A(\eta,\epsilon)$ around $p,$  
cf. Section~\ref{sec:perturb} and Lemma~\ref{lem:one}.
\begin{lemma}
	\label{lem:perturbation_several}
	Let $p \in M$ be a point on a compact
	manifold of dimension $n \ge 3$
	with a Riemannian metric $g$ or with
	a reversible Finsler metric $f.$
	Assume that $\eta>0$ satisfies $\eta<\inj/3,$
	here $\inj$ is the injectivity radius.
	Let $c_j:[-2\eta,2\eta]
	\longrightarrow M, j=1,2\ldots,N$ 
	be geodesic segments
	parametrized by arc length with $c_j(0)=p,$
	for which the initial directions
	for $j\not=k$ satisfy
	$v_j=c_j'(0)\not=\pm c_k'(0)=\pm v_k.$

	For sufficiently small $\epsilon>0$
	in any neigborhood of the metric $g$ resp. $f$
	with respect to the strong $C^r$-topolgy with
	$r\ge 2$ resp. $r\ge 4$ there is a Riemannian
	metric $\overline{g}$ resp. a reversible Finsler metric
	$\overline{f}$ with geodesic segments
	$\overline{c}_j:[-2\eta,2\eta]\longrightarrow
	M, j=1,\ldots,N$
	parametrized by arc length which coincide with $c_j$ 
	on the set $[-2\eta,-\eta-4\epsilon]
	\cup [\eta+4\epsilon,2\eta].$
	These geodesic segments do not intersect each other,
	i.e. $\overline{c}_j(S^1)\cap
	\overline{c}_k(S^1)=\emptyset.$
	The metrics $g$ and $\overline{g}$ resp.
	$f$ and $\overline{f}$ differ only
	on the union of the pairwise disjoint sets
	$U_j(\eta,\epsilon),
	j=1,\ldots,N.$ This is a subset of
	the spherical shell $A(\eta,\epsilon)$
	around $p.$
\end{lemma}
\begin{proof}
	Since $3\eta<\inj$ the geodesic segments
	are injective and the point $p$
	is the only point lying on distinct geodesic
	segments $c_j.$
	Then one can choose $\epsilon \in(0,\eta/6)$ 
	sufficiently small as in 
	Lemma~\ref{lem:one}
	such that 
	the maps $\xi_{v_j,w_j}$ defined in
	Equation~\eqref{eq:xi} are diffeomorphisms
	and
	the subsets
	$U_j(\eta,\epsilon)$ of the spherical shell
	$A(\eta,\epsilon)$ are pairwise disjoint.
	
	We can choose $\epsilon \in (0,\eta/6)$
	sufficiently small such that  the 
	intersections 
	$T_{v_j}(\eta,\epsilon)\cap A(\eta,\epsilon),
	j=1,2,\ldots,N$
	of the
	tubular neighborhoods
	$T_{v_j}(\eta,\epsilon)$ of the geodesic $c_j$ with
	the spherical shells $A(\eta,\epsilon)$ are pairwise disjoint.
	Since the sets $U_j(\eta,\epsilon)$ are subsets
	of the these intersections we conclude that also
	the sets $U_j(\eta,\epsilon),j=1,\ldots, N$ are 
	pairwise disjoint.
	
	If the dimension $n$ is at least four, we can find
	unit vectors $w_j,j=1,\ldots, N$ which are pairwise distinct
	such that 
	for sufficiently small $\eta>0, \epsilon \in (0,\eta/6)$
	the local surfaces $P_j=P(v_j,w_j), j=1,2,\ldots, N$ 
	in a neighborhood of $p$
	defined by $v_j,w_j,$ cf. Equation~\ref{eq:pvw}, pairwise
	only meet in the point $p.$ I.e.
	$(P_j-\{p\})\cap (P_k-\{p\})=\emptyset$
	for all $j,k, j\not=k.$ 
	
	If the dimension $n=3$ then the local surfaces
	$P_j,P_k$ 
	for sufficiently small $\epsilon>0$
	for distinct $j\not=k$ intersect
	in a smooth curve
	\begin{equation}
		\label{eq:intersection_curve}
		\gamma_{jk}(t)=\zeta_{v_j,w_j}(t_{jk}(u),x_{jk}(u),0)
		=\zeta_{v_k,w_k}(t_{kj}(u),x_{kj}(u),0)
	\end{equation}
	parametrized by arc length. 
	Here the curve is defined on a sufficiently small
	interval $[u_1,u_2]$ with $u_1<0<u_2$ such that 
	$t_{jk},t_{kj},x_{jk},x_{kj}:
	[u_1,u_2]\longrightarrow \R$ 
	are strictly monotone.
	This is possible since $\gamma_{jk}'(0)$ is neither a multiple
	of $c_j'(0)$ nor of $c_k'(0).$
	
	Without loss of generality we can assume that $v_j,w_j$
	are orthogonal to each other, i.e. $\langle v_j,w_j\rangle=0$
	for all $j=1,\ldots,N.$
	
	It follows from Lemma~\ref{lem:one} that for
	any $j=1,\ldots,N$ there is 
	a one-parameter family of Riemannian  $g_j^{(s)}, s\in [0,s_0], s_0>0$ 
	(resp.
	Finsler metrics)
	with $g_j^{(0)}=g$ which coincides with $g$
	on the complement of  $U_j^{\pm}(\eta,\epsilon)$
	and satisfies the following:
	The metrics $g_j^{(s)}$ have injective geodesics
	$c_{j,s}:[-2\eta,2\eta]\longrightarrow M, j=1,\ldots,N$
	parametrized by arc length which coincide with
	$c_j$ on $[-2\eta,-\eta-4\epsilon]\cup 
	[\eta+4\epsilon, 2\eta]$  and which
	are of the form
	$c_{j,s}(t)=\xi_{v_j,w_j}(t,u_{j,s}(t),0,\ldots,0)$
	for $t \in [-2\eta,2\eta]$ as described in
	Lemma~\ref{lem:one}. In particular the geodesic
	$c_{j,s}$ lies in the local surface
	$P_j$ and does not meet $p$ for $s>0,$
	and for $t \in [-\eta,\eta]$
	and sufficiently small positive $s:$
	\begin{equation}
		\label{eq:cjs}	
		c_{j,s}(t)=\xi_{v_j,w_j}(t,s,0,\ldots,0)\,.
	\end{equation}
	If $n\ge 4$ it follows that the
	the geodesics $c_{j,s}$ do not
	intersect pairwise for sufficiently small
	and positive $s,$
	since the local surfaces $P_j,P_k$ for
	distinct $j \not=k$ do only intersect in 
	the point $p.$
	If the dimension $n=3$ the intersection
	of $P_j,P_k$ is described above,
	cf. Equation~\eqref{eq:intersection_curve}. 
	Then we define for $s_1,\ldots,s_N >0$
	sufficiently small a metric
	$\overline{g}^{(s_1,\ldots,s_n)}$ by
	$g_j^{(s_j)}$ on the set
	$U_j(\eta,\epsilon)$ and by $g$ outside
	the union $U_1(\eta,\epsilon)\cup\ldots\cup
	U_N(\eta,\epsilon).$ From the form of the
	intersection of the local surfaces
	$P_j\cap P_k$ as described in 
	Equation~\eqref{eq:intersection_curve}
	and the form of the geodesics $c_{j,s}$
	in Equation~\eqref{eq:cjs} 
	we see that for a given $s>0$
	we can choose the parameters
	$s_1,s_2,\ldots,s_N\in (0,s)$ such that the
	geodesic segments $c_{j,s_j}(t)=\xi_{v_j,w_j}(t,s_j,0,), t\in [-\eta,\eta]$
	for distinct $j, k$ do not intersect.
	This is possible since for distinct
	$j,k$ there are unique $s^*_{jk},t^*_{j,k}$
	with $c_{j,s_j}(t)=c_{k,s^*_{jk}}(t^*_{jk}),$
	i.e. one has to choose $s_k\not=s^*_{jk},$
	for all $k\not=j,$
	cf. the description
	of the intersection of the local surfaces
	$P_j\cap P_k$ given in
	Equation~\eqref{eq:intersection_curve}.
\end{proof}
\begin{remark}
	\label{rem:nequals3}
	To understand the argument for dimension $n=3$ it may
	be helpful to consider the following 
	special case. Let
	$c_j:\R\longrightarrow \R^3,j=1,\ldots,N$ be 
	$N$ pairwise distinct straight lines $c_j$ in Euclidean $3$-space
	$\R^3$ intersecting in $p.$
	Then one can find $N$ pairwise distinct planes
	$E_j$ containing $c_j$ such that the intersection
	of the planes $E_j\cap E_k$ for distinct $j,k$ is a straight line
	$c_{jk}$ which is different
	from any of the lines $c_j.$ Then one can find for any $\epsilon >0$
	and any line
	$c_j, j=1,\ldots,N$ a parallel line $c_j'$ lying in $E_j$
	with distance $<\epsilon$ from $c_j$
	such that the intersection $c_j'\cap c_k'$ 
	for distinct $j,k$ is empty.
	In the Proof the curves $c_j''=c_{j,s_j}$ correspond then
	to curves which near $p$ are the straight segments $c_j'$ and outside
	a spherical shell $A(\eta,\epsilon)$ coincide with
	the original straight line $c_j.$ The case $N=2$ is 
	obvious,
	for $N\ge 3$ one has to choose the distances
	from $c_j$ and $c_j'$ to avoid intersections
	of the straight lines $c_j'.$ Then the Lemma implies
	that we can perturb the Euclidean metric on $\R^3$
	in a spherical shell around $p$
	such that the curves $c_j''$ become geodesics.
\end{remark}
\section{Proof of Theorem~\ref{thm:one}}
\label{sec:proof}	
\begin{proof}	
	Let $\mathcal{G}=\mathcal{G}^r(M)$ be the space of Riemannian metrics
	with the strong $C^r$-topology with $r\ge 2.$
	A closed geodesic $c$ is \emph{non-degenerate} if
	there is no periodic Jacobi field $Y=Y(t)$ which is orthogonal
	to the closed geodesic, i.e. $g(Y(t),c'(t))=0$ holds for all $t.$
	This also implies that $1$ is not
	an eigenvalue of the \emph{linearized Poincar\'e map}
	$P_c.$  
	We conclude from the \emph{bumpy metrics theorem}:
	For $a>0$ the set $\mathcal{G}(a)$ of Riemannian metrics
	for which all closed geodesics with length $\le a$
	are non-degenerate, is an open and dense subset of 
	$\mathcal{G}=\mathcal{G}(M).$
	
	It also follows that there are only finitely many geometrically
	distinct and prime closed geodesics $\tilde{c}_1,\ldots,\tilde{c}_r$	
	such that all closed geodesics of length $\le a$ are geometrically
	equivalent to one of the closed geodesics $\tilde{c}_j, j=1,\ldots,r.$
	
	Let $\mathcal{G}^*(a)\subset \mathcal{G}(a)$ 
	be the set of Riemannian metrics, such that
	all prime closed geodesics $\tilde{c}_1,\ldots,\tilde{c}_r$
	of length $\le a$ are simple and do not intersect
	each other. This is an open subset of the set $\mathcal{G}(a),$
	since there are only finitely many geometrically distinct
	closed geodesics of length $\le a$ in $\mathcal{G}^*(a),$
	cf. \cite[Lem.2.4]{Ra94}, \cite[\S 4]{An}.
	It remains to prove that the set $\mathcal{G}^*(a)
	\subset \mathcal{G}(a)$ is dense.
	
	It follows from Lemma~\ref{lem:intersection}
	that the union 
	$DIP(a)=DP(\tilde{c}_1)\cup\ldots\cup DP(\tilde{c}_r)
	\cup \bigcup_{j\not=k}I(\tilde{c}_j,\tilde{c}_k)$
	of the double points 
	$DP(\tilde{c}_j), j=1,\ldots,r$ of the prime closed 
	geodesics of length $\le a$ and
	the intersection points
	$I(\tilde{c}_j,\tilde{c}_k)$ for distinct $j,k$ is a finite
	set.
	Then we can find a sufficiently small 
	$\eta>0$ such that the geodesic
	balls $B_p(2\eta)$ for $p \in DIP(a)$
	are disjoint and such that the following holds: 
	
	For $p \in DIP(a)$ the  intersection
	$B_p(2\eta)\cap \left\{
	\tilde{c}_1(S^1)\cup \ldots \cup 
	\tilde{c}_r(S^1)
	\right\}$
	of the geodesic ball
	$B_p(2\eta)$ of radius $2\eta$
	around $p$ and the traces of the prime
	closed geodesics $\tilde{c}_1,\ldots,
	\tilde{c}_r$ consists only of 
	$N$ geodesic segments of length $4\eta$
	with midpoint $p.$ 
	These are geodesic segments of one of the
	closed geodesics $\tilde{c}_1,\ldots,\tilde{c}_r$
	if $p$ is a double point of this closed geodesic.
	If a geodesic $\tilde{c}_j$ enters the
	geodesic ball $B_p(2\eta),$
	i.e. if $c_j(s)\in B_p(2\eta)$
	for some $s \in S^1$
	then there is a parameter 
	$s_1$ with $|s-s_1|<2\eta$ such that
	$p=\tilde{c}_j(s_1).$
	
	By a linear change of the
	parametrization these $N$ geodesic segments are
	geodesic segments $c_1,\ldots,c_N:
	[-2\eta,2\eta]\longrightarrow M$
	parametrized by arc length with
	$c_j(0)=p$ and with  $v_j=c_j'(0), j=1,\ldots,N,$
	such that $v_j\not=\pm v_k$ for distinct $j\not=k.$
	
	These geodesic segments satisfy 
	the assumptions of
	Lemma~\ref{lem:perturbation_several}.
	Since the geodesic balls of radius $2\eta$
	around the points $p \in DIP(a)$ are
	disjoint we can apply 
	Lemma~\ref{lem:perturbation_several}
	for every point 
	$p \in DIP(a)$
	separately and change the
	metric in the geodesic ball $B_p(2\eta).$
	
	Hence we obtain in any neighborhood of $g\in 
	\mathcal{G}(a)$ a metric
	$\overline{g}$ with prime closed
	geodesics $\overline{c}_1,\ldots,\overline{c}_r.$
	The length $L(\overline{c}_j)$
	equals the length $L(c_j)$ of $c_j.$
	And for every
	$j=1,\ldots,N$ the closed geodesic $\tilde{c}_j$
	coincides with $c_j$ outside the
	union 
	$$
	\bigcup_{p\in DIP(a)} B_p(2\eta)
	$$
	of the geodesic balls of radius $2\eta$ around
	the finitely many points $p\in DIP(a).$
	
	In any sufficiently small neighborhood of 
	a metric $g \in \mathcal{G}(a)$ the number
	of closed geodesics of length $\le a$ cannot
	increase, since the closed geodesics seen
	as periodic orbits of the geodesic flow are
	non-degenerate, cf.~\cite[\S 4, i)]{An}.
	Therefore the geodesics 
	$\overline{c}_1,\ldots,\overline{c}_r$
	are the prime closed geodesics of length $\le a$ of
	the metric $\overline{g}$
	up to geometric equivalence.
	Since the geodesics $\overline{c}_j, j=1,\ldots,N$ are simple
	and do not intersect each other we have shown
	that $\overline{g}\in\mathcal{G}^*(a).$
	Therefore we have shown that in any neighborhood of
	the metric $g$ 
	there is a metric $\overline{g}\in\mathcal{G}^*(a).$
	
	Then the intersection
	\begin{equation}
		\mathcal{G}^*=\bigcap_{k \in \n}\mathcal{G}^*
		(k)
	\end{equation}
	is a residual subset. 
	The set $\mathcal{G}^*(k)$ is the set of 
	Riemannian metrics for which the finitely many
	geometrically distinct
	prime closed geodesics of length $\le k$
	are simple, do not intersect each other
	and are non-degenerate.
	Therefore $\mathcal{G}^*$ is the set of
	Riemannian metrics for which all closed geodesics are 
	non-degenerate and all prime closed geodesics are
	simple. In addition distinct closed geodesics do not
	intersect.
	
	The argument in the Finsler case is the same, we only have to
	use the 
	following bumpy metrics theorem for reversible Finsler metrics:
\end{proof}
\begin{theorem}[Bumpy metrics theorem for 
	\emph{reversible} Finsler metrics]
	\label{thm:rev}
	For a compact differentiable manifold a $C^r$-generic 
	\emph{reversible}
	Finsler for $r\ge 4$
	metric is bumpy.
\end{theorem}
\begin{proof}
	We consider the space $\mathcal{F}_{rev}^r(M)$
	of reversible Finsler metrics on the compact manifold
	$M$ with respect to the strong $C^r$-topology. 
	The only necessary modification in the proof of~\cite[Thm.3]{RT2020}
	is that we have to choose the function $\phi: \R^n \longrightarrow
	\R$ in addition to be even, i.e. $\phi(-y)=\phi(y)$ for all
	$y \in \R^n.$		
\end{proof}
\begin{remark}[Surfaces]
	\label{rem:surfaces}
	The case of dimension $2,$ i.e. surfaces, is 
	quite different.	
	One can show that for any Riemannian metric 
	and any reversible Finsler metric on a 
	closed surface there exists a simple closed geodesic.
	If the surface is not simply-connected one can see easily
	that the shortest non-contractible closed geodesic is 
	simple. If the surface is simply-connected it is the
	famous result by Lusternik and 
	Schnirelman~\cite{LS} that there
	exist three simple closed geodesics. For a detailed 
	recent proof,
	which also works for reversible Finsler metrics,
	see~\cite{DMMS}. 
	Calabi and Cao have shown that the 
	shortest closed geodesic of a convex surface is simple,
	cf.~\cite{CC}. For a convex surface with 
	a Riemannian metric with sectional
	curvature $K \ge \delta>0$ a simple closed geodesic
	has length $\le 2\pi/\sqrt{\delta},$ this result is due
	to Toponogov, cf. \cite[3.4.10]{Kl}.
	Hence a $C^r$-generic Riemannian metric 
	on $S^2$
	of positive curvature with $r\ge 2$
	has only finitely many 
	geometrically distinct,
	simple closed geodesics. 
	On any convex surface two closed geodesics intersect,
	this statement holds for Riemannian metrics as
	well as for reversible Finsler metrics,
	cf.\cite{Bryant}. On the other hand there are non-reversible Finsler
	metric of positive flag curvature with
	two simple closed geodesics which do not intersect,
	cf. \cite{Ra2017} and \cite{Bryant}.
	
	In case of negative curvature 
	on a surface of genus $g$ the number $N(t)$ of
	closed geodesics of length $\le t$ grows exponentially,
	whereas the number $N_1(t)$ of simple closed geodesics
	of length $\le t$
	grows polynomially of order $6g-6.$ 
	Mirzakhani has been able to compute the asymptotic growth rate
	for $N_1(t),$
	cf. \cite{EMM}.
\end{remark}

\end{document}